\documentclass{amsart}
 \usepackage{amssymb}
\usepackage{calligra}
\input xy
\xyoption{all}
\newtheorem{theorem}{Theorem}[section]

\theoremstyle{definition}
\newtheorem{definition}[theorem]{Definition}
\newtheorem{example}[theorem]{Example}

\theoremstyle{remark}
\newtheorem{remark}[theorem]{Remark}

\theoremstyle{notation}

\theoremstyle{proposition}
\newtheorem{proposition}[theorem]{Proposition}

\theoremstyle{corollary}
\newtheorem{corollary}[theorem]{Corollary}

\numberwithin{equation}{section}

\begin{document}

\title{Pre-Lie algebras in positive characteristic}


\author{Ioannis Dokas}
\address{}
\email{dokas@ucy.ac.cy}

\subjclass{17D25, 17B50, 18C15}



\keywords{Restricted Lie algebra, dendriform algebra, pre-Lie
algebra, algebras with divided powers over an operad}

\begin{abstract}
In prime characteristic we introduce the notion of restricted
pre-Lie algebras. We prove in the pre-Lie context the analogue to
Jacobson's theorem for restricted Lie algebras. In particular, we
prove that any dendriform algebra over a field of positive
characteristic is a restricted pre-Lie algebra. Thus we obtain that
Rota-Baxter algebras and quasitriangular algebras are restricted
pre-Lie algebras. Moreover, we prove that the free
$\Gamma(\calligra{preLie})$-algebra is a restricted pre-Lie algebra,
where $\calligra{preLie}$ denotes the pre-Lie operad. Finally, we
define the notion of restricted enveloping dendriform algebra and we
construct a left adjoint functor for the functor $(-)_{p-preLie}:
Dend \rightarrow p-preLie$.
 \end{abstract}

\maketitle





\section*{\textbf{Introduction}}

In many cases, in the theory of Lie algebras over a field of prime
characteristic, the notion of Lie algebras has to be replaced by the
notion of restricted Lie algebras. The theory of restricted Lie
algebras was developed by N. Jacobson in \cite{Jac}. Restricted Lie
algebras arise naturally in positive characteristic. Indeed, by a
theorem of N. Jacobson, the associated Lie algebra $A_{Lie}$ of any
associative algebra $A$ over a field of characteristic $p\neq 0$, is
endowed with the structure of a restricted Lie algebra
$(A_{Lie},[-,-],F)$, where $F$ denotes the Frobenius map.

The structure of pre-Lie algebras appear in many fields of
mathematics under the names of left symmetric algebras, right
symmetric algebras or Vinberg algebras. Pre-Lie algebras are
Lie-admissible algebras which generalize associative algebras. The
notion of a pre-Lie algebra has been introduced by Gerstenhaber in
\cite{Ger}. In connection to the theory of operads, if $\mathcal{P}$
is an operad the space $\bigoplus_{n} P(n)$ admits a structure of a
pre-Lie algebra \cite{LV}.

In this paper we study pre-Lie algebras in prime characteristic. In
the pre-Lie context the role of associative algebras is now played
by dendriform algebras introduced by J.-L. Loday in \cite{L3}.
Dendriform algebras are algebras with two binary operations, which
dichotomize the notion of associative algebras. Moreover, any
dendriform algebra structure, induces a pre-Lie algebra structure.
Therefore there is a functor $(-)_{preLie}: Dend \rightarrow preLie$
from the category of dendriform algebras to the category of pre-Lie
algebras.

In section $2$, in a similar way to the case of restricted Lie
algebras, we introduce the notion of \emph{restricted pre-Lie
algebra} generalizing the definition given by A. Dzhumadil'daev in
\cite{Dz}. We define a restricted pre-Lie algebra as a pair
$(P,\{-,-\},(-)^{[p]})$ where $P$ is a pre-Lie algebra together with
a map $(-)^{[p]}: P\rightarrow P$ which verifies specific relations.
From the definition of a restricted pre-Lie algebra
$(P,\{-,-\},(-)^{[p]})$, the induced Lie algebra $(P_{Lie},[-,-])$
is actually a restricted Lie algebra $(P_{Lie},[-,-],(-)^{[p]})$.
Therefore we obtain a functor from the category of restricted
pre-Lie algebras to the category of restricted Lie algebras.

Next, in Theorem $2.11$ we prove the analogue to Jacobson's theorem
in the pre-Lie context: Any dendriform algebra over a field of
characteristic $p\neq 0$ is a restricted pre-Lie algebra. Thus we
obtain a functor $(-)_{p-preLie}: Dend \rightarrow p-preLie$, from
the category of dendriform algebras to the category of restricted
pre-Lie algebras.

In section $3$, we relate restricted pre-Lie algebras with
$\Gamma(\calligra{preLie})$-algebras, where $\calligra{preLie}$
denotes the pre-Lie operad. In particular, if $\mathcal{P}$ is an
operad, then B. Fresse in \cite{Fre1} constructs the monad
$\Gamma(\mathcal{P})$ in $Vect$. Moreover, he defines the notion of
\emph{algebra with divided powers} as an algebra over the monad
$\Gamma(\mathcal{P})$. It follows from Cartan's theory (see
\cite{Car}, \cite{Fre1}), that the notion of
$\Gamma(\calligra{Com})$-algebras coincides with the notion of
commutative algebras with divided powers, where $\calligra{Com}$
denotes the commutative operad. If $\calligra{Lie}$ denotes the Lie
operad, then B. Fresse in \cite{Fre1} proves that a
$\Gamma(\calligra{Lie})$-algebra is nothing else than a restricted
Lie algebra. In section $3$ we prove that the free
$\Gamma(\calligra{preLie})$-algebra is a restricted pre-Lie algebra.

In section $4$, we give examples of restricted pre-Lie algebras. We
note that any example in the sense of the definition of A.
Dzhumadil'daev is an example of a restricted pre-Lie algebra. In
particular, Witt algebras are restricted pre-Lie algebras. Besides,
there is a problem of finding those Lie algebras which admit a
pre-Lie structure. In contrast to the characteristic $0$ case, D.
Burde proves in \cite{Bur1} that the simple Lie algebra
$\calligra{sl}(2,k)$, admits pre-Lie algebra structures if and only
if $char\,k=3$. We prove that the simple Lie algebra
$\calligra{sl}(2,k)$, actually admits a structure of a restricted
pre-Lie algebra. Moreover, we prove by Propositions $4.1$ and $4.2$
that Rota-Baxter algebras and quasitriangular algebras over a field
of positive characteristic are restricted pre-Lie algebras. In
section $5$ we define for $(P,\{-,-\},(-)^{[p]})$ a restricted
pre-Lie algebra, the notion of \emph{restricted enveloping
dendriform algebra} $U_{p}(P)$ and we construct a left adjoint
functor for the functor $(-)_{preLie}: Dend \rightarrow p-preLie$.

\bigskip

 \textbf{Acknowledgements.}\; I am grateful to the referee for important
 suggestions and comments.

\section{Restricted Lie algebras}

In order to study structures with additional operations (see for
example \cite{DL}), many ideas and inspiration come from the
category $p-Lie$ of restricted Lie algebras over a field $k$ of
characteristic $p\neq 0$. We recall Jacobson's definition of
restricted Lie algebras.

\begin{definition}
A \emph{restricted} Lie algebra $(L,(-)^{[p]})$ over a field $k$
of characteristic $p\neq 0$ is a Lie algebra $L$ over $k$ together with a
map $(-)^{[p]}: L \rightarrow L$ called the $p$-map such that the following relations
hold:
\begin{align}
(\alpha x)^{[p]} &=  \alpha^{p}\;x^{[p]}\\
[x^{[p]},y ]     &=  [\underbrace{x,[x,[\cdots [x}_{p},y]]]] \\
(x+y)^{[p]} &= x^{[p]}+y^{[p]}+\sum_{i=1}^{p-1}s_{i}(x,y)
\end{align}
where $i s_{i}(x,y)$ is the coefficient of $\lambda^{i-1}$ in the
formal $(p-1)$-fold product $$[\underbrace{\lambda x+y,[\cdots [\lambda x
+y}_{p-1},x]\cdots]]$$
where $x,y \in L$ and $\alpha \in k$.
\end{definition}

\begin{remark}
We note that if $(L,(-)^{[p]})$ is a restricted Lie algebra and $H$ denotes the Lie sub-algebra of $L$ generated by $x,y \in L$, then $s_{i}(x,y) \in H^{p}$
where $H^{i}$ is defined by $H^{1}=H$ and $H^{i}:=[H^{i-1},H]$.
\end{remark}

The next theorem due to N. Jacobson \cite{Jac} justifies the above definition.

\begin{theorem}
Let $(A,\star)$ be an associative algebra over a field $k$ of
characteristic $p\neq 0$. Then $(A,[-,-],F)$ is a restricted Lie
algebra where for all $x,y\in A$ we denote by $[x,y]:=x\star y-
y\star x$, and $F$ is the Frobenius map $F: A\rightarrow A$ given by
$F(x):=x^{\star p}$.
\end{theorem}

\section{Restricted pre-Lie algebras}

We recall that M. Gerstenhaber has defined the notion of pre-Lie
algebra in \cite{Ger}. In particular, the Lie bracket involved in
the Gerstenhaber structure on Hochschild cohomology comes from a
pre-Lie structure. Moreover, pre-Lie algebras appeared in E.
Vinberg's work on convex homogeneous cones \cite{Vi} and they play
an important role in the differential geometry of flat manifolds
\cite{Milnor2}. Besides, the theory of pre-Lie algebras has received
a great impulse, because of its applications to renormalization, as
formalized by A. Connes and D. Kreimer \cite{Co}. For a survey on
pre-Lie algebras the reader may consult the paper of D. Burde
\cite{Bur2}.

\begin{definition}
A \emph{left pre-Lie algebra} $(A,{-,-})$ is a $k$-vector space equipped with a binary operation $\{-,-\}: A\rightarrow A$ such that the following relation holds:
$$\{x,\{y,z\}\}-\{\{x,y\},z\}=\{y,\{x,z\}\}-\{\{y,x\},z\}.$$
\end{definition}

If we denote by $a(x,y,z)$ the associator then the above relation is
written $$a(x,y,z)=a(y,x,z).$$ For this reason in the literature,
left pre-Lie algebras are also called, left symmetric algebras. As
we mentioned, the category of pre-Lie algebras is Lie admissible. In
particular we have the following proposition.

\begin{proposition}
If $(A, \{-,-\})$ is a pre-Lie algebra, then $A$ is equipped with a
structure of a Lie algebra with a Lie bracket given by
$[x,y]:=\{x,y\}-\{y,x\}$.
\end{proposition}

\begin{proof}
See Lemma $1.4.2$ in \cite{LV}.
\end{proof}

\emph{Free pre-Lie algebras.} Let $X$ be a set, then F. Chapoton and
M. Livernet in \cite{CL} give an explicit description of the free
pre-Lie algebra $\rm{preLie}(X)$ generated by $X$ in terms of rooted
trees labeled by $X$.

Similarly to the case of restricted Lie algebras, we propose the
following definition for restricted pre-Lie algebras.

\begin{definition}
\emph{A restricted pre-Lie algebra} $(P,(-)^{[p]})$, is a pre-Lie algebra $P$ over a field $k$ of characteristic $p\neq 0$ together with a
map $(-)^{[p]}: P \rightarrow P$ called the $p$-map such that:
\begin{align}
(\alpha x)^{[p]}  &=  \alpha^{p}\;x^{[p]}\\
\{x^{[p]},y \}    &=  \{\underbrace{x,\{x,\{\cdots \{x}_{p},y\}\}\} \\
\{y,x^{[p]}\}     &=  \{\underbrace{x,\{x,\{\cdots \{x}_{p},y\}\}\}+[\underbrace{x,[x,[\cdots [x}_{p},y]]]]  \\
(x+y)^{[p]} &= x^{[p]}+y^{[p]}+\sum_{i=1}^{p-1}s_{i}(x,y)
\end{align}
where $[x,y]:=\{x,y\}-\{y,x\}$ is the induced bracket, $s_{i}(x,y)$
are defined as in the Definition $1.1$ and $x,y \in P,\;\alpha \in
k$.
\end{definition}

If $(P,(-)^{[p]})$ and $(P',(-)^{[p]})$ are restricted pre-Lie algebras, then a pre-Lie morphism $f: P\rightarrow P'$ is called restricted if $f(x^{[p]})=f(x)^{[p]}$. We denote by $p-\textrm{preLie}$ the category of restricted pre-Lie algebras over $k$.

From the definition above we see that if $(P,\{-,-\},(-)^{[p]})$ is a restricted pre-Lie algebra, then
the induced Lie algebra is a restricted Lie algebra $(P,[-,-],(-)^{[p]})$. Therefore there is a functor from the category of restricted pre-Lie algebras to
the category of restricted Lie algebras and we obtain the following proposition.

\begin{proposition}
The following diagram of categories of algebras is commutative
\[
\xymatrix{p-\rm{preLie} \ar[d]  \ar[r]  & p-\rm{Lie} \ar[d] \\
\rm{preLie} \ar[r]      & \rm{Lie} } \]
\end{proposition}

\begin{remark}
The Definition $2.3$ of restricted pre-Lie algebras, generalizes the
Definition $1.3$ given by A. Dzhumadil'daev in \cite{Dz}. Indeed, in
\cite{Dz} a restricted pre-Lie algebra is defined as a pre-Lie
algebra $(P,\{-,-\})$ such that
$$\{x^{\{p\}},y\}=\{\underbrace{x,\{x,\{\cdots \{x}_{p},y\}\}\},$$ where $$x^{\{p\}}:=\{\underbrace{x,\{x,\{\cdots \{x,x}_{p}\}\}\}\}.$$
Moreover, by Corollary $2.4$ in \cite{Dz}, it is proved that the
induced Lie algebra is a restricted Lie algebra
$(P,[-,-],(-)^{\{p\}})$. We easily see that the relations of the
Definition $2.3$ are verified, thus a restricted pre-Lie algebra $P$
in terms of the Definition $1.3$ in \cite{Dz} is a restricted
pre-Lie algebra $(P,\{-,-\},(-)^{[p]})$ in terms of the Definition
$2.3$ where  $x^{[p]}:=x^{\{p\}}$.
\end{remark}

Next, we give an example of a restricted pre-Lie algebra which is
not a restricted pre-Lie algebra in the sense of Dzhumadil'daev's
Definition $1.3$ in \cite{Dz}.

\begin{example}
Let $K$ be a field with $char\,K=2$ and $P:=Kx+Ky$ be the $K$-vector
space generated by two generators $x,y$. We endow $P$ with the
following  pre-Lie product:
\begin{align*}
\{x,x\} &= 0,  \;\; \{x,y\}=y\\
\{y,x\} &=0,    \;\; \{y,y\}=0 \\
\end{align*}
The induced Lie algebra $(P, [-,-])$ is endowed with the Lie bracket
$[x,y]=y$. Also, $P$ is a restricted Lie algebra with $p$-map such
that $x^{[p]}=x$ and $y^{[p]}=0$. The map is extended by the formula
$(ax+by)^{[p]}=a^{2}x+(ab)y$. Moreover,

\begin{align*}
 \{(ax+by), \{(ax+by),y\}\} &= \{(ax+by), ay\}\\
                            &=a^{2}y\\
                            &=\{a^{2}x+(ab)y, y\}
\end{align*}
and
\begin{align*}
 \{(ax+by), \{(ax+by),x\}\} &= 0\\
                            &=\{a^{2}x+(ab)y, x\}
\end{align*}
It follows that the restricted Lie algebra structure on $(P, [-,-])$
is induced by a restricted pre-Lie algebra structure. Besides,
$$\{\{x,x\},y\}=\{0,y\}=0$$ and $\{x,\{x,y\}\}=y$. Therefore the above
restricted pre-Lie algebra is not a restricted pre-Lie algebra in
the sense of the Definition $1.3$ in \cite{Dz}.
\end{example}

We have seen that the structure of associative algebras in prime
characteristic, is the prototype in order to define the notion of
restricted Lie algebras. The question that arises naturally is: what
structure in the context of pre-Lie algebras will play the role of
associative algebras? As we will see below this role is played by
the structure of dendriform algebras introduced by J.-L. Loday in
\cite{L3}. Let us recall the definition of a dendriform algebra.

\subsection{Dendriform algebras}

In many cases in combinatorial constructions we obtain associative products which have the nice property to split into components.
This phenomenon arises also in other contexts and the category of dendrifrom algebras play a crucial role.

\begin{definition}
A \emph{dendriform algebra} $(D,\prec,\succ)$ is a $k$-vector space endowed with two binary operations
$\prec:\; D \rightarrow D$ and $\succ\; : D\rightarrow D$ such that the following relations hold:

\begin{align*}
(x\prec y)\prec z          &=x\prec (y\prec z+y\succ z),\\
(x\succ y)\prec z          &=x\succ (y\prec z),\\
(x\prec y+x\succ y)\succ z &= x\succ(y\succ z).
\end{align*}
for all elements $x,y,z \in D$.
\end{definition}

If we denote by $\star : D\rightarrow D$ the operation given by $x\star y:= x\prec y+ x\succ y$ then by Lemma $5.2$ in \cite{L3} we have that $(D,\star)$ is
an associative algebra. The above relations are concisely written as follows:
\begin{align}
(x\prec y)\prec z  &=x\prec (y\star z),\\
(x\succ y)\prec z  &=x\succ (y\prec z),\\
(x\star y)\succ z  &=x\succ(y\succ z).
\end{align}

\emph{Shuffle algebra.} J.-L. Loday proves  in \cite{L3} that the
shuffle algebra is a dendriform algebra. Indeed, the shuffle product
can be split in two products in a way that the axioms of dendriform
algebra are satisfied.

\begin{remark}
A dendriform algebra $D$ for which $x\succ y=y\prec x$ for all $x,y
\in D$, is called Zinbiel algebra. The notion of Zinbiel algebras is
related to the notion of commutative algebras with divided powers
(see \cite{Dokas}).
\end{remark}

\emph{Free Dendriform algebras}. There exists an explicit
description of free dendriform algebras in terms of planar binary
trees (cf. Theorem $5.8$ in \cite{L3}).

The next proposition is well known and proves that there is a functor from the category of dendriform algebras to the category of pre-Lie algebras.

\begin{proposition}
Let $(D,\prec,\succ)$ be a dendriform algebra. Then $D$ is equipped with a left pre-Lie algebra operation $\{-,-\}$ given by
$\{x,y\}:=x\succ y-y\prec x$ for all $x,y \in D$.
\end{proposition}

Therefore there is a functor from the category of dendriform
algebras to the category of pre-Lie algebras.

\begin{remark}
If $(D,\prec, \succ)$ is a dendriform algebra, then the Lie
structures associated to $(D,\star)$ and $(D,\{-,-\})$ coincide.
\end{remark}

As a consequence, considering the associated categories of algebras we have the following proposition.

\begin{proposition}
The following diagram of categories of algebras is commutative

\[
\xymatrix{\rm{Dend} \ar[d]  \ar[r]  & \rm{preLie} \ar[d] \\
\rm{As} \ar[r]      & \rm{Lie} } \]
\end{proposition}

Next we prove a theorem analogue to the  Jacobson's Theorem $1.3$ in the pre-Lie algebra context.

\begin{theorem}
Let $(D,\prec,\succ)$ be a dendriform algebra over a field $k$ of characteristic $p\neq 0$. Then $(D,\{-,-\},(-)^{\star p})$ is a restricted pre-Lie algebra.
\end{theorem}

\begin{proof}
First we prove the relation $2.2$. Let $x,y\in D$. We denote by
$R_{x}$ the right $\prec$ action by $x$ namely $R_{x}(y):=y\prec x$.
Similarly, we denote by $L_{x}$ the left $\succ$ action by $x$, i.e
$L_{x}(y):=x\succ y$. From relation $2.6$ in the axioms of the
definition of dendriform algebras we obtain that
\begin{align*}
R_{x}(L_{x}(y)) &=(x\succ y)\prec x\\
                                &=x\succ(y\prec x)\\
                                &=L_{x}(R_{x}(y)).
\end{align*}
Besides, we have
\begin{align*}
\{x,y\} &=x\succ y-y\prec x \\
        &=L_{x}(y)-R_{x}(y)\\
        &=(L_{x}-R_{x})(y).
\end{align*}
Since $L_{x}$ and $R_{x}$ commute we obtain that $(L_{x}-R_{x})^{p}=\sum_{i=0}^{i=p}\binom{p}{i} (L_{x})^{i}(R_{x})^{p-i}$. Therefore in prime
characteristic $p$ we obtain that $(L_{x}-R_{x})^{p}=(L_{x})^{p}-(R_{x})^{p}$. Moreover from relation $2.7$ we have $(L_{x})^{p}(y)=x^{\star p}\succ y$. Besides
from relation $2.5$ we get $(R_{x})^{p}(y)=y \prec x^{\star p}$.

Finally, we have
\begin{align*}
(L_{x}-R_{x})^{p}(y) &=\{\underbrace{x,\{x,\cdots \{x}_{p},y\}\cdots \},\}\\
                                  &=(L_{x})^{p}-(R_{x})^{p}\\
                                  &=x^{\star p}\succ y-y \prec x^{\star p}\\
                                  &=\{x^{\star p},y\}.
\end{align*}

Now by Jacobson's Theorem $1.3$, the associative operation endows
$D$ with a Frobenius and therefore $(D,[-,-],(-)^{\star p})$ is a
restricted Lie algebra for the induced Lie bracket. Therefore the
relations $2.1$ and $2.4$ obviously follow.

Finally, we have $$[x^{\star p},y ]=[\underbrace{x[x[\cdots [x}_{p},y]]]]$$ and the relation $2.3$ is a consequence of
$[x^{\star p},y]=\{x^{\star p},y\}-\{y,x^{\star p}\}$.
\end{proof}

Therefore from Theorem $2.12$ we obtain, in prime characteristic,
the analogue of Proposition $2.11$.

\begin{proposition}
The following diagram of categories of algebras is commutative \
\[
\xymatrix{\rm{Dend} \ar[d]  \ar[r]  & p-\rm{preLie} \ar[d] \\
\rm{Ass} \ar[r]      & p-\rm{Lie} } \]
\end{proposition}

\begin{remark}
In prime characteristic, Turchin in \cite{Tour} proves in Theorem
$7.1$ that any brace algebra is equipped with the structure of a
restricted Lie algebra with $p$-map given by $a \mapsto a^{\{p\}}$.
Moreover, F. Chapoton in \cite{Chap} and M. Ronco \cite{Ron1} define
a functor of categories of algebras $\rm{Den} \rightarrow
\rm{Brace}$ lifting the functor $\rm{Den} \rightarrow \rm{preLie}$.
Therefore any dendriform algebra is equipped with the structure of a
restricted Lie algebra with $p$-map $a \mapsto a^{\{p\}}$. Thus if
$D\in \rm{Den}$ is a dendriform algebra the induced Lie algebra
$D_{Lie}$ is equipped with two restricted Lie algebra structures. It
follows that $a^{\star p}-a^{\{p\}} \in \mathcal{C}(D_{Lie})$, where
$a\in D_{Lie}$ and $\mathcal{C}(D_{Lie})$ denotes the center of
$D_{Lie}$.
\end{remark}

\section{Pre-Lie algebras with divided powers}

In this section we firstly recall some results from the theory of
operads and associated monads. For details, we refer the reader to
the book of J.-L. Loday and B. Vallette \cite{LV} and B. Fresse's
paper \cite{Fre1}. We prove that the free
$\Gamma(\calligra{preLie})$-algebra is a restricted pre-lie algebra.

First we recall the notion of a monad (cf. \cite{Mac}). A
\emph{monad} $(T,\eta,\gamma)$ in a category $\rm{C}$ consists of a
functor $T: \rm{C} \rightarrow \rm{C}$ and two natural
transformations $\eta: Id_{\rm{C}} \rightarrow T$ and $\gamma:
T\circ T \rightarrow T$ which make the following diagrams commute

\[
\begin{gathered}
\xymatrix{(T\circ T)\circ T \ar[d]^{\gamma T}  \ar[r]^{T\gamma}  & T\circ T \ar[d]^{\gamma} \\
T\circ T \ar[r]^{\gamma}      & T } \text{\qquad \qquad}
\xymatrix{IT \ar[rd] \ar[r]^{\eta T} &T\circ T \ar[d]^{\gamma}
&TI \ar[l]_{T\eta} \ar[ld]\\
&T}
\end{gathered}
\]
where $T\gamma: (T\circ T)\circ T \rightarrow T\circ T$ denotes the
natural transformation with components $(T\gamma)_{X}=T(\gamma_{X})$
and $\gamma T: (T\circ T)\circ T\rightarrow T\circ T$ has components
$(\gamma T)_{X}=\gamma_{T(X)}$, for each $X\in \rm{C}$. Moreover,
let $(T, \eta, \gamma)$ be a monad in a category $\rm{C}$. A
\emph{$T$-algebra} is an object $X\in \rm{C}$ together with an arrow
$h: T(X)\rightarrow X$ of $\rm{C}$ such that the following diagrams
commute

\[
\begin{gathered}
\xymatrix{(T\circ T)(X) \ar[d]^{\gamma_{X}}  \ar[r]^{Th}  & T(X) \ar[d]^{h} \\
T(X) \ar[r]^{h}      & X } \text{\qquad  \qquad} \xymatrix{ X
\ar[rd]^{id} \ar[r]^{\eta_{X}} & TX  \ar[d]^{h}
&    \\
&X}
\end{gathered}
\]

An $\mathbf{S}$-module $\mathcal{P}$ is a family of right
$S_{n}$-modules $\mathcal{P}(n)$ for $n\geq 0$. A morphism of
$\mathbf{S}$-modules is a family of $S_{n}$-equivariant maps. We
denote by $\mathbf{S}-Mod$ the category of $\mathbf{S}$-modules. To
any $\mathbf{S}$-module $\mathcal{P}$ is associated a functor
$T(\mathcal{P}): Vect \rightarrow Vect$ called \emph{Schur functor},
given by
$$T(\mathcal{P})(V):=\bigoplus_{n\geq
0}(\mathcal{P}(n)\otimes V^{\otimes n})_{S_{n}}$$

Moreover, B. Fresse in \cite{Fre1} defines the functor
$\Gamma(\mathcal{P}): Vect \rightarrow Vect$ given by

$$\Gamma(\mathcal{P})(V):=\bigoplus_{n\geq
0}(\mathcal{P}(n)\otimes V^{\otimes n})^{S_{n}}$$

In characteristic $0$, there is an isomorphism $T(\mathcal{P})(V)
\simeq \Gamma(\mathcal{P})(V)$ given by the norm map (cf.
\cite{Fre1}). There is a notion of \emph{composition} of
$\mathbf{S}$-modules (see $5.1.6$ in \cite{LV}). In particular, if
$M$ and $N$ are $\mathbf{S}$-modules, then the composite of $M$ and
$N$ is given by
$$M\circ N:=\bigoplus_{n\geq 0}M(n)\otimes_{S_{n}}N^{\otimes n}$$
By definition, an \emph{algebraic operad} is an $\mathbf{S}$-module
$\mathcal{P}$ equipped with an associative product $\gamma:
\mathcal{P}\circ \mathcal{P}\rightarrow \mathcal{P}$ together with
unit $\eta: I\rightarrow \mathcal{P}$ (see \cite{LV}). Let
$\mathcal{P}$ be an algebraic operad, then the associated Schur
functor $T(\mathcal{P}): \rm{Vect} \rightarrow \rm{Vect}$ is a monad
in $\rm{Vect}$. If $\mathcal{P}$ is an operad, then a
\emph{$\mathcal{P}$-algebra} is an algebra over the monad
$T(\mathcal{P})$. In characteristic $0$ all classical type of
algebras are algebras over the corresponding operad.

B. Fresse in $1.1.18$ \cite{Fre1} proves that $\Gamma(\mathcal{P})$
is a monad in $\rm{Vect}$ and he defines the notion of
\emph{$\mathcal{P}$-algebra with divided powers} as an algebra over
the monad $\Gamma(\mathcal{P})$. The norm map (cf. 1.1.14 in
\cite{Fre1}) is a morphism of monads $T(\mathcal{P})\rightarrow
\Gamma(\mathcal{P})$. Thus there is a forgetful functor from the
category of $\Gamma(\mathcal{P})$-algebras to the category of
$\mathcal{P}$-algebras.

In positive characteristic, the functors $T(\mathcal{P})$ and
$\Gamma(\mathcal{P})$ are different in general. Therefore the
notions of a $\mathcal{P}$-algebra and a $\mathcal{P}$-algebra with
divided powers are different in general. However, if
$\mathcal{P}(n)$ is $S_{n}$-projective, then the notions of a
$\mathcal{P}$-algebra with divided powers and a
$\mathcal{P}$-algebra, coincide. It follows that a
$\Gamma(\calligra{As})$-algebra and a
$\Gamma(\calligra{Dend})$-algebra don't carry more structure than an
associative algebra, and a dendriform algebra respectively. For the
case of the commutative operad $\calligra{Com}$ it follows from H.
Cartan's theory (see \cite{Car}, \cite{Fre1}) that a
$\Gamma(\calligra{Com})$-algebra is nothing but a commutative
algebra with divided powers. For the case of the Lie operad
$\calligra{Lie}$, B. Fresse in \cite{Fre1} proves that a
$\Gamma(\calligra{Lie})$-algebra is, nothing but a restricted Lie
algebra. A natural question arises about the structure of
$\Gamma(\calligra{preLie})$-algebras where $\calligra{preLie}$
denotes the pre-Lie operad.

\begin{proposition}
 The free $\Gamma(\calligra{preLie})$-algebra generated by a vector space is a restricted pre-Lie
 algebra.
\end{proposition}

\begin{proof}
We denote by $\Gamma(\calligra{preLie})(V)$ the free
$\Gamma(\calligra{preLie})$-algebra generated by a vector space $V$.
There is a commutative diagram of morphisms of operads

\[
\xymatrix{ \calligra{Dend\;}   & \calligra{preLie} \ar[l]\\
\calligra{As} \ar[u]      & \calligra{Lie}\ar[u] \ar@{_{(}->}[l] }
\]

Besides, F. Chapoton in \cite{Chap} and M. Ronco in \cite{Ron}
proved that the there is an injection of operads $\calligra{preLie}
\hookrightarrow \calligra{Dend}$. Moreover, F. Chapoton and M.
Livernet in \cite{CL} and M. Markl in \cite{Mar} study the morphism
of operads $\calligra{Lie} \rightarrow \calligra{preLie}$, (see also
\cite{Seg}). It is proved that, there also is an injection of
operads $\calligra{Lie} \hookrightarrow \calligra{preLie}$. Thus the
above diagram of operads induces  the following commutative diagram
of morphisms of monads
\[
\xymatrix{T(\calligra{Dend})(V)=\Gamma( \calligra{Dend})(V)   & \Gamma(\calligra{preLie})(V) \ar@{_{(}->}[l]\\
T(\calligra{As})(V)=\Gamma(\calligra{As})(V) \ar[u] &
\Gamma(\calligra{Lie})(V)\ar@{^{(}->}[u] \ar@{_{(}->}[l] }
\]
where $V\in Vect$. The norm map $T(\calligra{preLie})\rightarrow
\Gamma(\calligra{preLie})$ is a natural transformation and induces
on a $\Gamma(\calligra{preLie})$-algebra the structure of a
$T(\calligra{preLie})$-algebra. Moreover, we have the following
commutative diagram of morphisms of monads
 \[ \xymatrix{
T(\calligra{preLie}) \ar[rd] \ar[r]
&\Gamma(\calligra{preLie}) \ar@{_{(}->}[d]\\
& T(\calligra{Den})}
\]
Thus we get that $\Gamma(\calligra{preLie})(V)$ is a pre-Lie
sub-algebra of $T(\calligra{Den})(V)$. If $x\in V$ then by Lemma
$1.2.7$ in \cite{Fre1} we have that $x^{\star p}\in
\Gamma(\calligra{Lie})(V)\subset \Gamma(\calligra{preLie})(V)$.
Therefore from the above commutative square and Theorem $2.12$, we
obtain that $\Gamma(\calligra{preLie})(V)$ is endowed with the
structure of a restricted pre-Lie algebra.
\end{proof}

\section{Witt algebras, Rota Baxter algebras and infinitesimal bialgebras}

In this section, we give examples of restricted pre-Lie algebras. We
note that by Remark $2.5$, any example of a restricted pre-Lie
algebra in terms of the definition of A. Dzhumadil'daev in
\cite{Dz}, is an example of restricted pre-Lie algebra in terms of
the Definition $2.3$. In particular, Witt algebras are restricted
pre-Lie algebras.

\subsection{Witt algebras} In the late thirties  Witt gave an
example of a simple Lie algebra, which behaves completely
differently from classical simple Lie algebras. Let us recall the
definition of a Witt algebra.

Let $k$ be a field with $char\;k=p>3$, we consider the truncated
polynomial algebra $P:=k[X]/<X^{p}>$. Let $x$ be the image of $X$ in
$P$, the set $\{x^{i}, 0 \leq i < p\}$ forms a $k$-basis for $P$.
Let $W:=Der_{k}(P)$ be the set of derivations of the $k$-algebra
$P$. Then $W$ is a Lie subalgebra of $End_{k}(P)$ called the
\textit{Witt algebra}. Also, in prime characteristic, by the Leibniz
rule follows that $D^{p}$ is also a derivation. Thus the Witt
algebra $W$ is a restricted Lie subalgebra of $End_{k}(P)$. Besides,
if we denote by $d:=d/dX$, then $d(X^{p})=0$. Therefore $d$ induces
a derivation of $P$ which we denote by slight abuse of notation by
$d$. Let $e_{i}:=x^{i+1}d$, where $-1 \leq i \leq p-2$. Then it
follows that $W=\oplus_{-1 \leq i \leq p-2} ke_{i}$ and the Lie
bracket is given by:
\[
[e_{i},e_{j}]:=
\begin{cases}
(j-i)e_{i+j},\; -1 \leq i+j \leq p-2,\\
0, \; \text{otherwise}
\end{cases}
\]
Moreover, $W$ is endowed with the structure of a restricted Lie
algebra with $p$-map given by
\[
e_{i}^{[p]}=
\begin{cases}
e_{0},\; \text{if}\; i=0\\
0, \; \text{otherwise}
\end{cases}
\]

The restricted Lie algebra structure on $W$ is induced by a
restricted pre-Lie algebra structure, where the pre-Lie product is
given by:
\[
\{e_{i},e_{j}\}:=
\begin{cases}
(j+1)e_{i+j},\; -1 \leq i+j \leq p-2,\\
0, \; \text{otherwise}
\end{cases}
\]

and the $p$-map is such that $e_{i}^{[p]}=e_{i}^{\{p\}}$.

\subsection{The simple Lie algebra $\calligra{sl}$(2,k)}
The structure of pre-Lie algebras arise in the theory of affine
manifolds. One natural question is whether the Lie algebra of simply
connected Lie group admits a pre-Lie structure. The problem of
finding those Lie algebras which admit a structure of a pre-Lie
algebra has been studied  and references can be found in
\cite{Bur1}, \cite{Bur2}. In characteristic $0$, if $\mathbf{g}$ is
a semisimple algebra, it follows from Whitehead's lemma that
$\mathbf{g}$ does not admit any pre-Lie structure. This is not true
in prime characteristic. If $char\;k=p>2$, then D. Burde proves in
Proposition $2.1.1$ that, the classical simple Lie algebra
$\calligra{sl}(2,k)$ admits pre-Lie structures if and only if $p=3$.
In particular, let $char\;k=3$ and $\calligra{sl}(2,k)=kx\oplus ky
\oplus kz$ with Lie products given by $[x,y]=z,[z,x]=2x$ and
$[z,y]=-2y$. Then in \cite{Bur1} is given a pre-Lie structure for
$\calligra{sl}(2,k)$ defined by the following formulas

\begin{align*}
\{x,x\} &=0    &\{y,x\} &=0    &   \{z,x\} &=0    \\
\{x,y\} &=-2z  & \{y,y\} &=0   &   \{z,y\} &=2y               \\
\{x,z\} &=x    &  \{y,z\} &=y  &   \{z,z\} &=z
\end{align*}

The simple Lie algebra $\calligra{sl}(2,k)$ is a restricted Lie
algebra with $p$-map given by $$(\lambda x+\mu y+\kappa
z)^{[p]}=(\lambda \kappa^{2}+\lambda^{2}\mu)x+(\lambda \mu^{2}+\mu
\kappa^{2})y+ (\kappa^{3}+\lambda \kappa\mu)z$$

A simple calculation shows that this restricted Lie algebra
structure admits  a restricted pre-Lie algebra structure. In
particular, we have
\begin{align*}
\{(\lambda x\underbrace{+\mu y+\kappa z),\cdots,\{(\lambda x+\mu y+}_{3}\kappa z),x\}\}\} &=0\\
                                    &=\{(\lambda x+\mu y+\kappa
z)^{[p]},x\}
\end{align*}
Also,
\begin{align*}
 \{(\lambda x\underbrace{+\mu y+\kappa z),\cdots,\{(\lambda x+\mu y+}_{3}\kappa z),y\}\}\}= &\{(\lambda x+\mu y+\kappa
z), \{(\lambda x+\mu y+\kappa z),(\lambda z-\kappa y)\}\}\\
=&\{(\lambda x+\mu y+\kappa z),(\lambda^{2}x+(\mu\lambda+\kappa^{2})y\}\\
=&(\lambda^{2}\mu+\lambda \kappa^{2})z-(\mu\lambda\kappa+\kappa^{3})y\\
=&\{(\lambda x+\mu y+\kappa z)^{[p]},y\}
\end{align*}
Besides,
\begin{align*}
\{(\lambda x\underbrace{+\mu y+\kappa z),\cdots,\{(\lambda x+\mu
y+}_{3}\kappa z),z\}\}\}=& \{(\lambda x+\mu y+\kappa z),\{(\lambda
x+\mu y+\kappa z),(\lambda x+\mu y+ \kappa z)\}\}\\
=& \{(\lambda x+\mu y+\kappa z),(\lambda\mu+\kappa^{2})z+\lambda\kappa x \}\\
=&(\lambda \kappa^{2}+\lambda^{2}\mu)x+(\lambda \mu^{2}+\mu
\kappa^{2})y+ (\kappa^{3}+\lambda \kappa\mu)z\\
=&\{(\lambda x+\mu y+\kappa z)^{[p]},z\}
\end{align*}

Moreover, as a consequence of Theorem $2.12$ we obtain a large
amount of examples of restricted pre-Lie algebras.

\subsection{Rota-Baxter algebras} Gian-Carlo Rota introduced special types of operators (see in \cite{R1}) in the category of associative algebras
beyond the usual derivations. In particular, let $(A,\cdot)$ be an
associative algebra then an operator $\beta: A\rightarrow A$ is
called \emph{Rota-Baxter} operator if $$\beta (x) \cdot \beta
(y):=\beta(\beta(x)\cdot y+x\cdot\beta(y)).$$ An associative algebra
$(A,\beta)$ equipped with a Rota-Baxter operator $\beta$ is called a
Rota-Baxter algebra. Rota-Baxter algebras appear in many fields of
theoretical physics and mathematics.

\begin{proposition}
A Rota Baxter algebra $(A,\cdot,\beta)$ over a field $k$ of characteristic $p\neq 0$ is equipped with the structure of a restricted pre-Lie algebra.
\end{proposition}

\begin{proof}
From Aguiar' s Proposition $4.5$ in \cite{Ag} any Rota Baxter
algebra is equipped with the structure of a dendriform algebra with
operations given by $$x\succ y:=\beta(x)\cdot y$$ and $$x\prec
y:=x\cdot \beta(y).$$ Therefore by Theorem $2.12$ the Rota Baxter
algebra $(A,\beta)$  is equipped with the structure of a restricted
pre-Lie algebra $(A,\{-,-\},(-)^{\cdot p})$ where the pre-Lie
bracket is given by $\{x,y\}:=\beta (x)\cdot y-y\cdot \beta(x)$.
\end{proof}

\subsection{Quasitriangular algebras} Infinitesimal algebras were introduced by S. A. Joni and Gian-Carlo Rota in \cite{JR}. An infinitesimal bialgebra or an $\epsilon$-bialgebra
is a triple $(A,\mu,\Delta)$ where $(A,\mu)$ is an associative
algebra, $(A,\Delta)$ is a coassociative coalgebra, and $\Delta$ is
a derivation. A special class of $\epsilon$-bialgebras are called
quasitriangular $\epsilon$-bialgebras. M. Aguiar in \cite{Ag} proves
that, there is a functor from the category of quasitriangular
$\epsilon$-bialgebras to the category of dendriform algebras. Thus
by Theorem $2.12$ we obtain the following proposition:

\begin{proposition}
Any quasitriangular $\epsilon$-bialgebra over a field $k$ of characteristic $p\neq 0$ is equipped with the structure of a restricted pre-Lie algebra.
\end{proposition}

\section{Restricted enveloping dendriform algebra}

It is known (see section $3$ in \cite{Fre2}, \cite{Jones}) that
given a morphism of operads $\mathcal{P} \rightarrow \mathcal{Q}$
the ``forgetful functor"  from $\mathcal{Q}$-algebras to
$\mathcal{P}$-algebras admits a left adjoint given by a coequalizer.
As an application of this general result (see e.g. Proposition $7$
in \cite{Chap}) one obtains the enveloping dendriform algebra of a
pre-Lie algebra and a left adjoint to the functor
$(-)_{\textrm{preLie}}: Dend\rightarrow preLie.$

\begin{definition}
Let $P$ be a pre-Lie algebra. A universal enveloping dendriform
algebra of $P$ is a pair $(U,i)$, where $U$ is a dendriform algebra,
$i: P \rightarrow U$ is a pre-Lie homomorphism and the following
holds: for any dendriform algebra $A$ and any pre-Lie homomorphism
$f: P\rightarrow A_{preLie}$ there exists a unique dendriform
homomorphism $\theta: U\rightarrow A$ such that $\theta i=f$.
\end{definition}

Let $P$ be a pre-Lie algebra we consider the free dendriform algebra
$Dend(P)$ generated by the vector space $P$. Let $\mathcal{R}$ be
the dendriform ideal of $Dend(P)$ generated by the elements: $$\{
\{x,y\}-(x\succ y-y\prec x)\}$$ where $x,y \in P$. We denote by
$U(P):=Dend(P)/\mathcal{R}$. From the general theory follows the
next statement.

\begin{proposition}
The functor $U: \textrm{preLie} \rightarrow \textrm{Dend}$ is left adjoint to the
functor $(-)_{\textrm{preLie}}: Dend\rightarrow preLie$. Therefore we have:
$$Hom_{preLie}(P,A_{preLie})\simeq Hom_{Dend}(U(P),A)$$
\end{proposition}

Similarly to the notion of restricted enveloping algebra in the Lie
context, we define the notion of \emph{restricted enveloping
dendriform algebra} for a restricted pre-Lie algebra. Thus we
construct a left adjoint functor $$U_{p}: p-\textrm{preLie}
\rightarrow \textrm{Dend}$$ to the functor
$$(-)_{p-\textrm{preLie}}: \textrm{Dend} \rightarrow
p-\textrm{preLie}$$

\begin{definition}
Let $P$ be a restricted pre-Lie algebra. A universal restricted
enveloping dendriform algebra of $P$ is a pair $(U_{p},i)$, where
$U_{p}$ is a dendriform algebra, $i: P \rightarrow U_{p}$ is a
restricted pre-Lie homomorphism and the following holds: for any
dendriform algebra $A$ and any restricted pre-Lie homomorphism $f:
P\rightarrow A_{p-preLie}$ there exist a unique dendriform
homomorphism $\theta: U_{p}\rightarrow A$ such that $\theta i=f$.
\end{definition}

Let $(P,(-)^{[p]})$ be a restricted pre-Lie algebra, we consider the
free dendriform $Dend(P)$ algebra generated by the vector space $P$.
Let $\mathcal{R}_{p}$ be the dendriform ideal of $Dend(P)$ generated
by the elements: $$\{ \{x,y\}-(x\succ y-y\prec x), x^{[p]}-x^{\star
p}\}$$ where $x,y \in P$. We denote by
$U_{p}(P):=Dend(P)/\mathcal{R}_{p}$. Let $i:P \rightarrow U_{p}(P)$
be the restriction to $P$ of the canonical homomorphism of $Dend(P)$
onto $U_{p}(P)$.

\begin{proposition}
Let $(P,(-)^{[p]})$ be a restricted pre-Lie algebra. The pair
$(U_{p}(P),i)$ is a universal restricted enveloping dendriform
algebra for $P$.
\end{proposition}

\begin{proof}
Let $f: P\rightarrow A_{p-preLie}$ be a restricted pre-Lie
homomorphism, where $A$ is a dendriform algebra. From the universal
property of $U(P)$ there is a dendriform homomorphism $\theta':U(P)
\rightarrow A$ which extends $f$. Moreover we have:

\begin{align*}
\theta'(x^{[p]})-\theta' (x^{\star p})&=f(x^{[p]})-\theta'(x)^{\star p}\\
&=f(x)^{\star p}-\theta'(x)^{\star p}\\
&=0
\end{align*}
The homomorphism $\theta'$ induces a unique dendriform homomorphism
$\theta: U_{p}(P)\rightarrow A$ such that $\theta i=f$.
\end{proof}

\begin{corollary}
The functor $U_{p}: p-\textrm{preLie} \rightarrow \textrm{Dend}$ is left adjoint to the
functor $(-)_{p-\textrm{preLie}}$. Therefore we have:
$$Hom_{p-preLie}(P,A_{p-preLie})\simeq Hom_{Dend}(U_{p}(P),A).$$
\end{corollary}

\bibliographystyle{plain}

\end{document}